\documentclass{amsart}
\usepackage[margin=1.5in]{geometry}
\usepackage{mathabx}
\usepackage{mathrsfs}
\usepackage{amsfonts}
\usepackage{lineno,hyperref}
\usepackage{amsmath}
\usepackage{amssymb}
\usepackage{amsthm}
\usepackage{cancel}
\usepackage{tikz-cd}
\usepackage{arydshln}
\usepackage{graphicx}
\usepackage{cite}
\usepackage{algorithm}
\usepackage[noend]{algpseudocode}

\newcommand{\Z}{\mathbb{Z}}

\newcommand{\Ulam}{\mathscr{U}}
\newcommand{\Sierp}{\mathcal{S}}
\newcommand{\concat}{{}^\frown}

\theoremstyle{plain}
\newtheorem{theorem}{Theorem}[section]
\newtheorem{definition}[theorem]{Definition}
\newtheorem{corollary}[theorem]{Corollary}
\newtheorem{conjecture}{Conjecture}[section]
\newtheorem{lemma}[theorem]{Lemma}

\theoremstyle{remark}

\newtheorem{remark}{Remark}[section]

\begin{document}
\title{Distributions of Ulam Words up to Length 30}

\author[Adutwum]{Paul Adutwum}
\address{Bates College}
\email[Paul Audutwum]{padutwum@bates.edu}

\author[Clark]{Hopper Clark}
\address{Bates College}
\email[Hopper Clark]{hclark2@bates.edu}

\author[Emerson]{Ro Emerson}
\address{Bates College}
\email[Ro Emerson]{remerson@bates.edu}

\author[Sheydvasser]{Alexandra (Sasha) Sheydvasser}
\address{University of Massachusetts, Amherst}
\email[Sasha Sheydvasser]{asheydvasser@umass.edu}

\author[Sheydvasser]{Arseniy (Senia) Sheydvasser}
\address{Department of Mathematics, Bates College}
\email[Senia Sheydvasser]{ssheydvasser@bates.edu}

\author[Tougouma]{Axelle Tougouma}
\address{Bates College}
\email[Axelle Tougouma]{atougouma@bates.edu}

\date{\today}

\begin{abstract}
We further explore the notion of Ulam words considered by Bade, Cui, Labelle, and Li, giving some lower bounds on how many there are of a given length. Gaps between words and words of special type also reveal remarkable structure. By substantially increasing the number of computed terms, we are also able to sharpen some of the conjectures made by Bade et al.
\end{abstract}

\maketitle

\section{Introduction:}\label{section: introduction}

In their 2020 paper\cite{BCLL_2020}, Bade, Cui, Labelle, and Li introduced the notion of Ulam words, defined as follows. Consider the free semigroup $S[\{0,1\}]$ on two generators $0$ and $1$. We say that $0$ and $1$ are \emph{Ulam} and then define all other Ulam words inductively: a word $w \neq 0,1$ is Ulam if and only if there exists exactly one pair of Ulam words $u_1 \neq u_2$ such that $w = u_1 \concat u_2$. (Here, $\concat$ denotes concatenation.) We shall denote the entire set of Ulam words as $\Ulam$, and Ulam words of length $n$ by $\Ulam_n$. It is easy to check that:
    \begin{align*}
        \Ulam_1 &= \{0,1\} & \Ulam_3 &= \{001,011,100,110\} \\
        \Ulam_2 &= \{01,10\} & \Ulam_4 &= \{0001,0010,0100,0111,1000,1011,1101,1110\}.
    \end{align*}
Bade et al. computed all Ulam words up to length $24$; we were able to compute up to length $30$. While this might appear as a small improvement at first glance, because the number of Ulam words of length $n$ appears to (almost) double on each iteration, in reality, this represents nearly $60$ times as much data.

Of course, it is an open question whether $\#\Ulam_n$ grows exponentially. The best general bound that we can prove is linear, mainly using explicit constructions of words from Bade et al.

\begin{theorem}\label{theorem: linear growth rate}
For all $n \geq 6$, $\#\Ulam_n \geq 2n+4$.
\end{theorem}

However, we are able to demonstrate that there is a {\it subsequence} of Ulam words that grows exponentially, using a completely different argument.

\begin{theorem}\label{theorem: exponential growth rate}
There exists $1 < \alpha_0 \leq 2$ such that for all $1 < \alpha < \alpha_0$,
    $$
    \limsup_{n\rightarrow\infty} \frac{\#\Ulam_n}{\alpha^n} = \infty.
    $$
Concretely,
    $$
    \alpha_0 = \left(\frac{101847671}{31}\right)^{1/5} \approx 1.648996
    $$
suffices.
\end{theorem}

We give proofs of both of these theorems in Section \ref{section: growth}. Unfortunately, both of these results are still quite far from what is conjectured to hold. To wit, define the density
    $$
    \rho(n) := \frac{\# \Ulam_n}{2^n};
    $$
Bade et al. conjectured that $\rho(n)\rightarrow r$ for some $0 < r < 1$ (see Conjecture 3.10 in \cite{BCLL_2020}); with our enlarged data set, we instead posit something a little stranger.

\begin{conjecture}\label{conjecture: growth rate}
$\rho(n)=\Theta(n^{-3/10})$.
\end{conjecture}

This conjecture is supported by the numerical evidence---see Figures \ref{fig:density_plot} and \ref{fig:raw count} for an example---but it also has ties to another conjecture involving the average gap between Ulam words, which we shall describe below. In any case, observe that if either conjecture is correct, the number of Ulam words grows only very slightly slower than $2^n$.

\begin{figure}[t]
    \centering
    \includegraphics[width=0.5\linewidth]{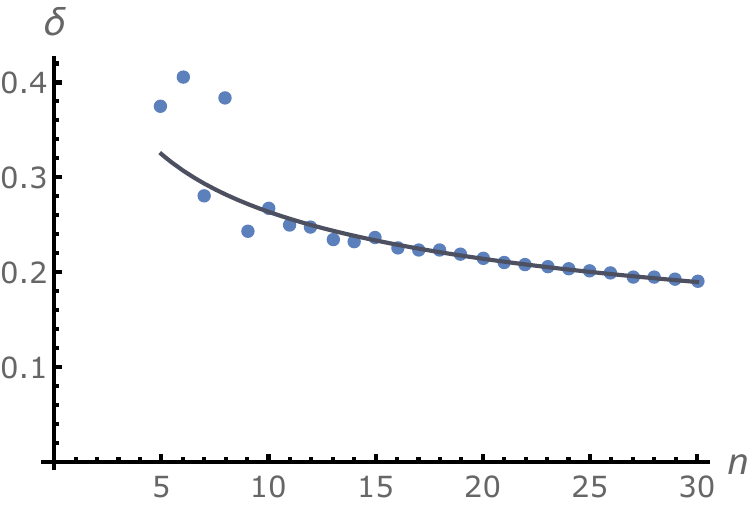}
    
    \caption{A plot of the densities $\rho(n)$ for $4 \leq n \leq 30$, together with a plot of $f(n) = 0.526 n^{-3/10}$.}
    \label{fig:density_plot}
\end{figure}

\begin{figure}
    \centering
    \begin{minipage}{0.3\textwidth}
    \begin{align*}
        \begin{array}{l|l}
        n & \Ulam_n \\ \hline
        13 & 1916 \\
        14 & 3812 \\ 
        15 & 7772 \\ 
        16 & 14822 \\ 
        17 & 29368 \\
        18 & 58478
    \end{array}
    \end{align*}
    \end{minipage}%
    \begin{minipage}{0.3\textwidth}
    \begin{align*}
        \begin{array}{l|l}
        n & \Ulam_n \\ \hline
        19 & 114300 \\ 
        20 & 225166 \\ 
        21 & 441724 \\ 
        22 & 876238 \\ 
        23 & 1717748 \\
        24 & 3406884
    \end{array}
    \end{align*}
    \end{minipage}%
    \begin{minipage}{0.3\textwidth}
    \begin{align*}
        \begin{array}{l|l}
         n & \Ulam_n \\ \hline
         25 & 6720784 \\ 
         26 & 13303332 \\ 
         27 & 26273948 \\ 
         28 & 52010642 \\ 
         29 & 102933200 \\ 
         30 & 203695342
    \end{array}
    \end{align*}
    \end{minipage}
    \caption{The exact counts for $\#\Ulam_n$ for $13 \leq n \leq 30$.}
    \label{fig:raw count}
\end{figure}

This notion of Ulam words was built on the earlier notion of Ulam sets due to Kravitz and Steinerberger\cite{kravitz_steinerberger_2017}, which was itself a generalization of Ulam's eponymous integer sequence, also defined recursively\cite{ulam_1964}: the (classical) Ulam sequence begins with $1,2$, and then every subsequent term is the next smallest integer that can be written as the sum of two distinct prior terms in exactly one way. Generalizations of Ulam's classic sequence have become an increasingly popular object of study: in 1972, Queneau did some preliminary work studying generalizations where the initial two terms of the integer sequence are varied\cite{queneau_1972}; in the 1990s, Cassaigne, Finch, Shmerl, and Spiegel determined some of the families of such sequences such that the consecutive differences are eventually periodic \cite{finch_1991,finch_1992_1,finch_1992_2,schmerl_spiegel_1994,cassaigne_finch_1995}; in 2017, Kravitz and Steinerberger considered generalizing the Ulam condition for abelian groups\cite{kravitz_steinerberger_2017}; in 2020, Bade et al. gave the aforementioned notion of Ulam words with some preliminary results\cite{BCLL_2020}; and in 2021, Sheydvasser showed that there is an analogous notion of Ulam sets for integer polynomials \cite{sheydvasser_2021} by building off earlier work of Hinman, Kuca, Schlesinger, and Sheydvasser\cite{HKSS_2019_1,HKSS_2019_2}.

Earlier work around Ulam words has largely centered around giving simple criteria for when words of some special type are Ulam---for example, Bade et al. showed that a word of the form $0^a 1 0^b$ is Ulam if and only if $\left(\begin{smallmatrix}a + b \\ a \end{smallmatrix}\right)$ is odd\cite{BCLL_2020}. Similarly, Mandelshtam considered Ulam words of the form $0^a 10^b 10^c$ and demonstrated a connection to the Sierpi\'nski gasket\cite{Mandelshtam_2022}. We also prove a few such results, such as the following.

\begin{theorem}\label{theorem: sierpinski}
Consider the set of points $(x,y) \in \Z_{\geq 1}^2$ such that $1^y 0^{x - y} \in \Ulam$. This is the discrete Sierpi\'nski triangle, union a point.
\end{theorem}

We will discuss this construction more precisely in Section \ref{section: unconditional results}, but briefly, the discrete Sierpi\'nski triangle is an approximation to the standard Sierpi\'nski triangle. It can be constructed either iteratively (as in Figure \ref{fig: Making S}) or by coloring Pascal's triangle by parity.

    \begin{figure}
        \centering
    \begin{tikzpicture}[scale=0.5]
        \filldraw[thick, black] (0.6,0.6) rectangle (1.4,1.4);
        \node at (1,-0.5) {$S_0$};
    \end{tikzpicture}
    \hspace{1cm}
    \begin{tikzpicture}[scale=0.5]
        \draw[thick, black] (0.6,0.6) rectangle (1.4,1.4);
        \filldraw[thick, black] (1.6,0.6) rectangle (2.4,1.4);
        \filldraw[thick, black] (1.6,1.6) rectangle (2.4,2.4);
        \node at (1.5,-0.5) {$S_1$};
    \end{tikzpicture}
    \hspace{1cm}
    \begin{tikzpicture}[scale=0.5]
        \draw[thick, black] (0.6,0.6) rectangle (1.4,1.4);
        \draw[thick, black] (1.6,0.6) rectangle (2.4,1.4);
        \draw[thick, black] (1.6,1.6) rectangle (2.4,2.4);
        \foreach\x in {0,2}{
            \filldraw[thick, black] (2.6,0.6+\x) rectangle (3.4,1.4+\x);
            \filldraw[thick, black] (3.6,0.6+\x) rectangle (4.4,1.4+\x);
            \filldraw[thick, black] (3.6,1.6+\x) rectangle (4.4,2.4+\x);
        }
        \node at (2.5,-0.5) {$S_2$};
    \end{tikzpicture}
    \hspace{1cm}
    \begin{tikzpicture}[scale=0.5]
        \draw[thick, black] (0.6,0.6) rectangle (1.4,1.4);
        \draw[thick, black] (1.6,0.6) rectangle (2.4,1.4);
        \draw[thick, black] (1.6,1.6) rectangle (2.4,2.4);
        \draw[thick, black] (2.6,0.6) rectangle (3.4,1.4);
        \draw[thick, black] (3.6,0.6) rectangle (4.4,1.4);
        \draw[thick, black] (3.6,1.6) rectangle (4.4,2.4);
        \draw[thick, black] (2.6,2.6) rectangle (3.4,3.4);
        \draw[thick, black] (3.6,2.6) rectangle (4.4,3.4);
        \draw[thick, black] (3.6,3.6) rectangle (4.4,4.4);
        \foreach\x in {0,4}{
            \filldraw[thick, black] (4.6,0.6+\x) rectangle (5.4,1.4+\x);
            \filldraw[thick, black] (5.6,0.6+\x) rectangle (6.4,1.4+\x);
            \filldraw[thick, black] (5.6,1.6+\x) rectangle (6.4,2.4+\x);
            \filldraw[thick, black] (6.6,0.6+\x) rectangle (7.4,1.4+\x);
            \filldraw[thick, black] (7.6,0.6+\x) rectangle (8.4,1.4+\x);
            \filldraw[thick, black] (7.6,1.6+\x) rectangle (8.4,2.4+\x);
            \filldraw[thick, black] (6.6,2.6+\x) rectangle (7.4,3.4+\x);
            \filldraw[thick, black] (7.6,2.6+\x) rectangle (8.4,3.4+\x);
            \filldraw[thick, black] (7.6,3.6+\x) rectangle (8.4,4.4+\x);
        }
        \node at (4.5,-0.5) {$S_3$};
    \end{tikzpicture}
    \caption{Visual of the first 4 steps of constructing the discrete Sierpi\'nski triangle.}
    \label{fig: Making S}
    \end{figure}
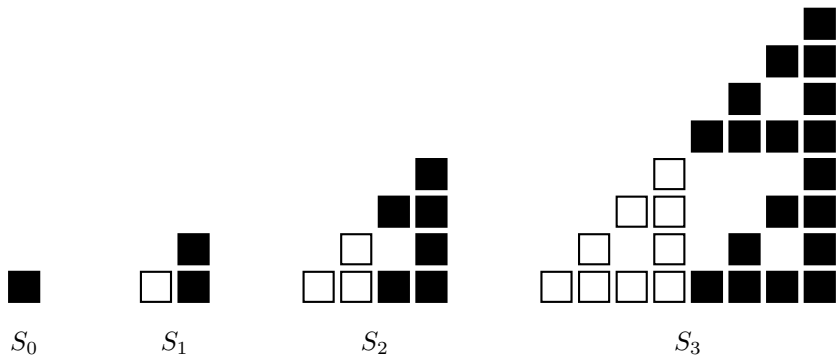

On the other hand, we also have a novel way of considering Ulam words by interpreting them as integers. Observe that there exists a natural map $\pi:S[\{0,1\}] \rightarrow \Z_{\geq 0}$ via interpreting a word as the binary representation of an integer. In general, this map is not injective---for example, $\pi(0) = \pi(00) = \pi(000) = 0$. However, if we restrict it to words of a fixed length, then it is. In particular, the restrictions $\pi:\Ulam_n \rightarrow \Z \cap [0,2^{n-1}]$ are injective maps. This gives a natural ordering on $\Ulam_n$ and allows us to ask questions about how Ulam words are distributed. For example, we might ask about the distribution of the gaps---differences between consecutive Ulam words, interpreted as integers.

\begin{conjecture}\label{conjecture: gap distribution}
Let $u_1 < u_2 < \ldots < u_{k_n}$ be the (ordered) elements of $\pi(\Ulam_n)$. Define
    \begin{align*}
    p_n: \Z_{\geq 1} &\rightarrow [0,\infty) \\
    g &\mapsto \frac{\# \left\{i\middle|u_{i + 1} - u_i = g\right\}}{k_n - 1}.
    \end{align*}
This has a natural interpretation as a probability measure. As $n \rightarrow \infty$, the functions $p_n$ converge pointwise to a probability measure $p:\Z_{\geq 1} \rightarrow [0,\infty)$. Furthermore, let $\mu_g(n)$ be the mean of the probability measure $p_n$. Then $\mu_g(n)=\Theta(n^{3/10})$---indeed, it may be that there is a constant $c\approx 1.9$ such that $\mu_g(n) = cn^{3/10} + o(1)$.
\end{conjecture}

This conjecture is well-supported by our available data---see Section \ref{section: archimedean equidistribution} for details, illustrations, and further odd properties of the apparent distribution. What is interesting about this statement about average gaps is that, if true, it immediately implies Conjecture \ref{conjecture: growth rate}.

\begin{theorem}\label{theorem: tie between conjectures}
As $n \rightarrow \infty$, $\rho(n)^{-1} \asymp \mu_g(n)$. Consequently, Conjecture \ref{conjecture: gap distribution} implies Conjecture \ref{conjecture: growth rate}.
\end{theorem}

This is salient, since our numerical evidence for Conjecture \ref{conjecture: gap distribution} is arguably much stronger than for Conjecture \ref{conjecture: growth rate}! Again, see Section \ref{section: archimedean equidistribution} for details. Finally, in Section \ref{section: non-archimedean equidistribution}, we ask the question of how $\pi(\Ulam_n)$ is distributed modulo $N$.

\begin{conjecture}\label{conjecture: nonarchimedean equidistribution}
For any integer $N > 1$ and $a \in \Z/N\Z$, define the \emph{relative density}
    $$
    \rho_{a,N}(n) := \frac{\# \left\{w \in \Ulam_n\middle|\pi(w) = a \mod N\right\}}{\# \Ulam_n}.
    $$
Then $\lim_{n\rightarrow\infty}\rho_{a,N}(n) = 1/N$.
\end{conjecture}

\begin{remark}
As we discuss in Section \ref{section: non-archimedean equidistribution}, while this conjecture is consistent with the available data, it is somewhat surprising. For one thing, $\rho_{5,6}(1) = \rho_{5,6}(2) = \rho_{5,6}(3) = 0$, and it takes some time before it appears to start to converge to $1/6$. For another, there is an apparent bias modulo $6$ in the distribution of the gaps.
\end{remark}

Our code and some of our data can be found on GitHub\footnote{https://github.com/asheydva/Ulam-Words.git}, but it is far from efficient---as was pointed out to us Tom\'as Oliveira e Silva, it is possible to use bitmaps to make these computations much faster; a good implementation should give a $O(2^n \log(n))$ running time. However, we leave this as material for future work.

\section{Definitions and Visualizations}\label{section: definitions}

We start with some basic definitions and constructions. Given a word $w \in S[\{0,1\}]$, we define its \emph{complement} $\hat{w}$ to be the word with every instance of $0$ replaced with a $1$, and vice versa. We also define the \emph{reverse} $\overline{w}$, which is the word obtained by reversing the order of the letters. It was shown by Bade et al. \cite{BCLL_2020} that $w \in \Ulam$ if and only if $\hat{w} \in \Ulam$, if and only if $\overline{w} \in \Ulam$.

To better visualize the set $\Ulam$, we made use of heat maps, which depict each Ulam word as a colored bar and stacks all of the words vertically---that is, for a given word, a $0$ corresponds to a rectangle of one color, and a $1$ corresponds to a rectangle of a second color. An example is provided in Figure \ref{fig:heat_map}. In general, we abridge such diagrams: we created figures only using all the Ulam words that started with zero, since Ulam words are closed under complements. Moreover, we impose the ordering discussed in the introduction, defining $w \leq w'$ if and only if $\pi(w) \leq \pi(w')$. Using this way of visualizing Ulam words allows us to easily see that there is both a clear binary tree structure that governs the existence of Ulam words, as well as a chaotic element to the set where the binary tree breaks down.

\begin{figure}
    \centering
    \includegraphics[width=0.8\linewidth]{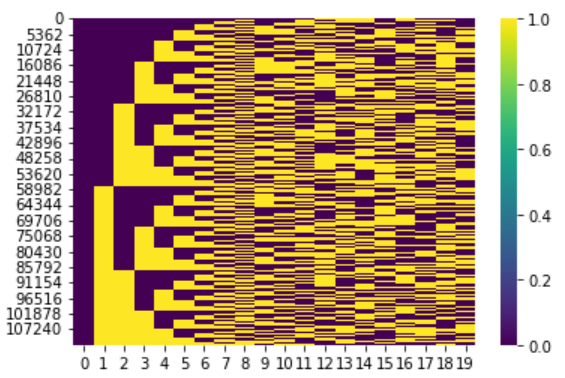}
    \caption{All words of length $20$ beginning with a zero.}
    \label{fig:heat_map}
\end{figure}

We can be more specific about our meaning regarding this breakdown: since Ulam words are preserved under the reverse map, this is equivalent to saying that for any $n$ there exists $n > \ell_n > 0$ such that all possible subwords of length $\ell_n$ occur as the final $\ell_n$ characters of words in $\Ulam_n$. In turn, that is equivalent to saying that the quotient map $\pi(\Ulam_n) \rightarrow \Z/2^{\ell_n}\Z$ is surjective. Our observation is that $\ell_n$ appears to increase as a function of $n$, albeit not very quickly---see Figure \ref{fig:surjectivity_data}. Assuming that Conjecture \ref{conjecture: nonarchimedean equidistribution} is true, then it would follow immediately that $\ell_n \rightarrow \infty$ simply by considering the case where $N = 2^{\ell_n}$---indeed, the heat maps were the original impetus for our equidistribution conjectures. On the other hand, the ``chaotic" latter half of the heat map is more of a mystery.

\begin{figure}[t]
    \begin{center}
    \begin{minipage}[t]{0.15\textwidth}
    \begin{align*}
        \begin{array}{c|c}
            n & \ell_n \\ \hline
            1 & 1 \\
            2 & 1 \\
            3 & 1 \\
            4 & 1 \\
            5 & 3 \\
            6 & 2
        \end{array}
    \end{align*}
    \end{minipage}%
    \begin{minipage}[t]{0.15\textwidth}
    \begin{align*}
        \begin{array}{c|c}
            n & \ell_n \\ \hline
            7 & 4 \\
            8 & 4 \\
            9 & 4 \\
            10 & 4 \\
            11 & 4 \\
            12 & 5
        \end{array}
    \end{align*}
    \end{minipage}%
    \begin{minipage}[t]{0.15\textwidth}
    \begin{align*}
        \begin{array}{c|c}
            n & \ell_n \\ \hline
            13 & 5 \\
            14 & 5 \\
            15 & 6 \\
            16 & 7 \\
            17 & 7 \\
            18 & 8
        \end{array}
    \end{align*}
    \end{minipage}%
    \begin{minipage}[t]{0.15\textwidth}
    \begin{align*}
        \begin{array}{c|c}
            n & \ell_n \\ \hline
            19 & 9 \\
            20 & 9 \\
            21 & 9 \\
            22 & 10 \\
            23 & 10 \\
            24 & 11
        \end{array}
    \end{align*}
    \end{minipage}%
    \begin{minipage}[t]{0.15\textwidth}
    \begin{align*}
        \begin{array}{c|c}
            n & \ell_n \\ \hline
            25 & 11 \\
            26 & 11 \\
            27 & 12 \\
            28 & 12 \\
            29 & 13 \\
            30 & 13
        \end{array}
    \end{align*}
    \end{minipage}
\end{center}
    \caption{Tables of $n$ versus $\ell_n$, where $\ell_n$ is the largest integer such that $\Ulam_n \rightarrow \Z/2^{\ell_n}\Z$ is surjective.}
    \label{fig:surjectivity_data}
\end{figure}

\section{Lower Bounds on Growth}\label{section: growth}

Our goal in this section is to prove our lower bounds on $\#\Ulam_n$; we begin with Theorem \ref{theorem: linear growth rate}, for which we need some explicit examples of Ulam words. The first three are due to Bade et al.

\begin{theorem}[Corollary 3.4 in \cite{BCLL_2020}]\label{theorem: bade 4}
There are $G(n-1)$ Ulam words of length $n$ of the form $0^a10^b$, where $G(n)$ is the $n$-th entry in Gould's sequence.
\end{theorem}

\begin{remark}
Gould's sequence $G(n)$ is the number of odd entries in the $n$-th row of Pascal's triangle; equivalently, $G(n)=2^{\#_1(n)}$, where $\#_1(n)$ is the number of non-zero bits in the binary representation of $n$.
\end{remark}

\begin{remark}
Since $w \in \Ulam$ if and only if $\overline{w} \in \Ulam$, if and only if $\hat{w} \in \Ulam$, we get analogous results with $1$'s replaced with $0$'s and the order of the letters reversed. This is true for all the results that we prove here.
\end{remark}

\begin{theorem}[Theorem 3.5 in \cite{BCLL_2020}]\label{theorem: bade 2}
For any $a,b \in \Z_{\geq 0}$, $0^a 1^2 0^b \in \Ulam$ if and only if the length of the word is odd (that is, $a + b = 1 \pmod 2$).
\end{theorem}

\begin{theorem}[Theorem 3.6 in \cite{BCLL_2020}]\label{theorem: bade 3}
For any $a,b \in \Z_{\geq 0}$ such that $a + b \geq 2$, $0^a 101 0^b \in \Ulam$ if and only if the length of the word is even (that is, $a + b = 1 \mod 2$).
\end{theorem}

\begin{lemma}\label{lemma: four ones}
For any $a,b\in\Z_{\geq 0}$ such that $a+b \geq 1$, $0^a 1^4 0^b \in \Ulam$ if and only if $a+b = 1 \pmod 4$.
\end{lemma}
\begin{proof}
We will use proof by induction on the length of the word $n$, where the base cases $n=5,6,7,8$ can be verified directly. Assume the statement holds for all words of length strictly less than $n$, and consider the word $u=0^k 1^4 0^l$ of length $n$, where, since $n\geq9$, at least one of $k\geq3$ or $l\geq3$. By applying the reverse map to switch $k$ and $l$ if necessary, we may assume that $k \geq 3$.

\begin{description}
    \item[Case 1]  $l = 0$.
        
        \noindent The only possible representations are $0\concat 0^{k-1}1^4$ and $0^k1^3\concat 1$. By the inductive hypothesis, the first is valid if and only if $n = 2 \pmod 4$. By Lemma \ref{lemma: sierpinski level 1}, the second is valid if and only if $n = 1,2 \pmod 4$. Thus, exactly one of these representations is valid if and only if $n = 1 \pmod 4$.
    \item[Case 2] $l \geq 1$.

    \noindent There are five potential representations:
        \begin{enumerate}
            \item $0\concat 0^{k-1}1^40^l$,
            \item $0^k1\concat 1^30^l$,
            \item $0^k1^2\concat 1^20^l$,
            \item $0^k1^3\concat 10^l$, and
            \item $0^k1^40^{l-1}\concat 0$.
        \end{enumerate}
    Observe that by the inductive hypothesis, representations (1) and (5) are valid if and only if $n = 2 \pmod 4$, which is to say that $k + l = 2 \pmod 4$. By Theorem \ref{theorem: bade 1} and Lemma \ref{lemma: sierpinski level 1}, representation (2) is valid if and only if $l = 0,3\pmod 4$; similarly, representation (4) is valid if and only if $k = 0,3\pmod 4$. Finally, by Theorem \ref{theorem: bade 2}, representation (3) is valid if and only if $k = l = 1 \pmod 2$. This allows us to count the number of valid representations in terms of the congruence classes of $k$ and $l$ modulo $4$.
    $$
    \begin{array}{c|cccc}
         k\backslash l & 0 & 1 & 2 & 3 \\ \hline
         0 & 2 & 1 & 3 & 2 \\
         1 & 1 & 3 & 0 & 2 \\
         2 & 3 & 0 & 0 & 1 \\
         3 & 2 & 2 & 1 & 5
    \end{array}
    $$
    In particular, there is a unique representation if and only if $n = 1 \pmod 4$.
\end{description}
\end{proof}

With this, we are ready to give a proof of the general, linear bound.

\begin{proof}[Proof of Theorem \ref{theorem: linear growth rate}]
We split into three cases.
\begin{description}
    \item[Case 1] $n$ is even.

    By Theorem \ref{theorem: bade 3}, we know that $0^a 101 0^{n-a-3} \in \Ulam_n$ for all $0 \leq a \leq n - 3$---this yields $n-2$ Ulam words. By Theorem \ref{theorem: bade 4}, we also know that there are $G(n-1)$ Ulam words of length $n$ of the form $0^a10^b$. Note that these two sets of Ulam words do not intersect (they have different numbers of ones), and since $n-1$ is odd, $G(n-1) \geq 2^2 = 4$. In total, this yields $n + 2$ Ulam words.

    Note that $\widehat{0^a1010^{n-a-3}}=1^a0101^{n-a-3} = 0^{a_1}10^{b_1}$ if and only if $n=3$, so the reverses of the constructed Ulam words are also distinct Ulam words. Therefore, we have at least $2n + 4$ Ulam words in this case.
    
    \item[Case 2] $n = 3 \pmod 4$.

    By Theorem \ref{theorem: bade 2}, we know that $0^a 1^2 0^{n-a-2} \in \Ulam_n$ for all $0 \leq a \leq n - 2$---this yields $n-1$ Ulam words. Since $n-1 = 2 \pmod 4$, $G(n-1) \geq 2^2 = 4$, and so we can again use Theorem \ref{theorem: bade 4} to conclude that there are at least $4$ words $0^a10^b$ of the right length. In total, this yields $n+3$ Ulam words.

    Note that $\widehat{0^a 1^20^{n-a-2}} = 1^a0^2 1^{n-a-2}=0^{a_1}10^{b_1}$ if and only if $a=0$ and $a_1=2$. Therefore, the reverses of our two families of constructed Ulam words intersect, but only in two places; therefore, we have $2(n+3)-2 = 2n + 4$ Ulam words.
    
    \item[Case 3] $n = 1 \pmod 4$.

    As in the previous case, we have $n-1$ words of the form $0^a 1^20^{n-a-2}$, but it is possible that $G(n-1) = 2$, so we have to argue differently: specifically, we use Lemma \ref{lemma: four ones} to conclude that $0^a 1^4 0^{n-4-a} \in \Ulam$ for all $0 \leq a \leq n-4$, which yields another $n-3$ Ulam words, for a total of at least $2n-4$.

    Observe that $\widehat{0^a1^20^{n-a-2}}\neq 1^a0^21^{n-a-2}=0^{a_1}1^40^{n-4-a_1}$ ever, so we may simply double our count of Ulam words. In total, we have $4n-8$, which is at least $2n+4$ if $n \geq 6$.
\end{description}

In each case, we have identified at least $2n+4$ distinct Ulam words.
\end{proof}

Next, we tackle the exponential bound, which we approach in a completely different fashion using the following lemma.

\begin{lemma}\label{lemma: weak bound}
For any $n \in \Z_{\geq 1}$,
    $$
    \# \Ulam_n^2 \leq \#\Ulam_n + \#\Ulam_{n+1}+ \ldots + \#\Ulam_{2n}.
    $$
\end{lemma}

\begin{proof}
Consider the set
    $$
    X:=\left\{(w_1,w_2)\in \Ulam_n^2\middle|w_1\neq w_2\right\}.
    $$
For any $(w_1,w_2)\in X$, either $w_1\concat w_2 \in \Ulam_{2n}$ or there exists $v_1 \in \Ulam_k$, $v_2 \in \Ulam_{2n-k}$ such that $w_1 \concat w_2 = v_1 \concat v_2$, where $k \in [1,n-1]\cup[n+1,2n]$; of course, if $k \in [1,n-1]$, then $2n-k \in [n+1,2n]$, and so we may conclude that
    $$
    \# X \leq \#\Ulam_{n+1}+ \ldots + \#\Ulam_{2n}.
    $$
On the other hand,
    $$
    \# X = \#\Ulam_n^2 - \# \Ulam_n.
    $$
\end{proof}

As a consequence, we get the following very weak lower bound: for any $n \in \Z_{\geq1}$,
    \begin{align}
    \max_{n \leq i \leq 2n} \# \Ulam_i \geq \frac{\#\Ulam_n^2}{n+1} \label{eq: weak lower bound}.
    \end{align}
This is sufficient for our purposes.

\begin{proof}[Proof of Theorem \ref{theorem: exponential growth rate}]
Choose any $n_1 \in \Z_{\geq 6}$ and let
    $$
    \alpha_0 := \left(\frac{\#\Ulam_{n_1}}{2(n_1+1)}\right)^{1/n_1}.
    $$
Then $\#\Ulam_{n_1}=2(n_1+1)\alpha_0^{n_1}$. By Theorem \ref{theorem: linear growth rate}, we know that $\#\Ulam_{n_1} > 2n_1 + 2$, hence $1 < \alpha_0 < 2$. Ergo, for any $1 < \alpha < \alpha_0$, $\#\Ulam_{n_1} = 2C(n_1+1)\alpha^{n_1}$ for some $C > 1$. Recursively define a sequence $n_1,n_2,\ldots$ such that $n_i$ is the unique integer $n_{i-1} \leq n_i \leq 2n_{i-1}$ such that $\#\Ulam_{n_i}$ is as large as possible. We'll prove by induction that $\#\Ulam_{n_k}\geq 2C^{2^k}(n_k+1)\alpha^{n_k}$. The base case $k=1$ is clear; for the induction step, we can apply our bound (\ref{eq: weak lower bound}) to see that
    \begin{align*}
    \Ulam_{n_{k+1}}&=\max_{n_k \leq i \leq 2n_k} \# \Ulam_i \geq \frac{\#\Ulam_{n_k}^2}{n_k+1} \\
    &> \frac{4C^{2^{k+1}}(n_k+1)^2}{n_k+1}\alpha^{2n_k} \\
    &> 2C^{2^{k+1}}(2n_k+1)\alpha^{2n_k} \\
    &\geq 2C^{2^{k+1}}(n_{k+1}+1)\alpha^{n_{k+1}}.
    \end{align*}
Since $C^{2^k}\rightarrow\infty$, we can conclude that for all $c > 0$, there exist infinitely many $n$ such that $\#\Ulam_n \geq c \alpha^n$. Therefore,
    $$
    \limsup_n \frac{\#\Ulam_n}{\alpha^n} = \infty,
    $$
as was claimed. All that remains is to find a value of $\alpha_0$ that works. In our case, we used $n_0=30$, which gives the $\alpha_0$ in the statement of the theorem.
\end{proof}

As can be seen from the proof of this theorem, if it is true that $\#\Ulam_n = \Theta(n^{-3/10}2^n)$ as we conjecture, then it will be possible to improve the constant $\alpha_0$ arbitrarily close to $2$ simply by computing $\#\Ulam_n$ for larger and larger $n$. Unfortunately, this rapidly becomes impractical and in any case will never suffice to prove the theorem with $\alpha_0=2$, as we suspect must be true.

\section{Patterns in Ulam Words}\label{section: unconditional results}

Bade et al. proved various results regarding Ulam words containing certain patterns \cite{BCLL_2020}. We offer a similar collection of results. We begin with a couple of intermediate results that allow us to prove Theorem \ref{theorem: sierpinski}.

\begin{theorem}[Theorem 3.4 in \cite{BCLL_2020}]\label{theorem: bade 1}
A word of the form $0^a 1 0^b$ is Ulam if and only if
    $$
    \left(\begin{smallmatrix}
        a + b \\ a
    \end{smallmatrix}\right) = 1 \mod 2.
    $$
\end{theorem}

\begin{lemma}\label{lemma: sierpinski level 1}
For any $n \geq 3$, $1^30^{n-3} \in \Ulam$ if and only if $n=0 \pmod 4$ or $n=1 \pmod 4$.
\end{lemma}
\begin{proof}
We will use proof by induction, where the base cases $n=3,4,5,6$ can all be verified by direct computation. Assume the statement holds for all words of length strictly less than $n$, and consider the word $u=1^30^{n-3}$. The only two possible representations for $u$ are
    \begin{enumerate}
        \item $1\concat 1^20^{n-3}$ and
        \item $1^30^{n-4}\concat 0$,
    \end{enumerate}
since neither $1^a$ nor $0^a$ are Ulam words for any $a > 1$. By Theorem \ref{theorem: bade 2}, $1^2 0^{n-3} \in \Ulam$ if and only if $n = 0 \pmod 2$; thus, representation (1) is valid if and only if $n = 0,2 \pmod 4$. On the other hand, by the inductive hypothesis, representation (2) is valid if and only if $n = 1,2 \pmod 4$. Ergo, exactly one of the representations is valid if and only if $n = 0,1 \pmod 4$.
\end{proof}

The proof of Lemma \ref{lemma: sierpinski level 1} illustrates how the modular length restrictions for $1^{a+1}0^b$ can easily be found using the length restrictions for $1^{a}0^b$, which, in turn, means we could generate countless additional theorems, providing length restrictions for $1^40^b$, $1^50^b$, and so forth. However, this would quickly prove tedious. Instead, we will demonstrate the unifying pattern between all words of the form $1^a0^b$. Recursively define
    \begin{align*}
        S_0 :&= (2,1) \\
        S_{n + 1} :&= S_n \cup \left(S_n + (2^n,0)\right) \cup \left(S_n + (2^n,2^n)\right) \\
        \Sierp :&= \bigcup_{n = 0}^\infty S_n.
    \end{align*}
The set $\Sierp$ is sometimes referred to as the discrete Sierpi\'nski triangle. The reason for this is that if one considers a suitable limit of the sets $2^{-n}\Sierp$, the result is the usual Sierpi\'nski triangle. Remarkably, this is exactly the correct construction to determine whether a word $1^a 0^b$ is Ulam or not.

\begin{theorem}
    $$
    \left\{(x,y) \in \Z_{\geq 1}^2\middle|1^y 0^{x-y} \in \Ulam\right\} = (1,1) \cup \Sierp.
    $$
\end{theorem}

\begin{remark}
Observe that this is Theorem \ref{theorem: sierpinski}, but stated more precisely.
\end{remark}

\begin{proof}
It is easy to check that the only elements in both sets with $x \leq 2$ are $(1,1)$ and $(2,1)$. Our goal is to show that the iterative process for producing points in $\Sierp$ with larger $x$ values is the same as for the first set. To do so, we will use a modified version of Pascal's triangle, achieved by aligning all entries to the left and then rotating the resulting image counterclockwise 90 degrees.
    \begin{center}
        \begin{tikzpicture}[scale=0.75]
            \foreach\x in {0,...,4}{
            \node at (\x/2,-\x) {1};
            \node at (-\x/2,-\x) {1};
            \draw[thick, black] (\x/2-0.25,-\x-0.25) rectangle (\x/2+0.25,-\x+0.25);
            \draw[thick, black] (-\x/2-0.25,-\x-0.25) rectangle (-\x/2+0.25,-\x+0.25);
            }
            \foreach\x in {2,3,4}{
            \node at (\x/2-1,-\x) {\x};
            \node at (-\x/2+1,-\x) {\x};
            }
            \node at (0,-4) {6};
            \draw[thick, black] (-0.75,-3.25) rectangle (-0.25,-2.75);
            \draw[thick, black] (0.75,-3.25) rectangle (0.25,-2.75);
            \node at (0,-4.5) {\vdots};
            \node [rotate=25] at (2.25,-4.5) {\vdots};
            \node [rotate=-25] at (-2.25,-4.5) {\vdots};
            \node at (3,-2.5) {$\Rightarrow$};
        \end{tikzpicture}
        \begin{tikzpicture}[scale=0.75]
            \node at (-1,0) {};
            \foreach\x in {0,...,4}{
            \node at (0,-\x) {1};
            \node at (\x,-\x) {1};
            \draw[thick, black] (-0.25,-\x-0.25) rectangle (0.25,-\x+0.25);
            \draw[thick, black] (\x-0.25,-\x-0.25) rectangle (\x+0.25,-\x+0.25);
            }
            \foreach\x in {2,3,4}{
            \node at (1,-\x) {\x};
            \node at (\x-1,-\x) {\x};
            }
            \node at (2,-4) {6};
            \draw[thick, black] (0.75,-3.25) rectangle (1.25,-2.75);
            \draw[thick, black] (1.75,-3.25) rectangle (2.25,-2.75);
            \node at (0,-4.5) {\vdots};
            \node [rotate=45] at (4.5,-4.5) {\vdots};
            \node at (5,-2.5) {$\Rightarrow$};
        \end{tikzpicture}
        \begin{tikzpicture}[rotate=90, scale=0.75]
            \foreach\x in {0,...,4}{
            \node at (0,-\x) {1};
            \node at (\x,-\x) {1};
            \draw[thick, black] (-0.25,-\x-0.25) rectangle (0.25,-\x+0.25);
            \draw[thick, black] (\x-0.25,-\x-0.25) rectangle (\x+0.25,-\x+0.25);
            }
            \foreach\x in {2,3,4}{
            \node at (1,-\x) {\x};
            \node at (\x-1,-\x) {\x};
            }
            \node at (2,-4) {6};
            \draw[thick, black] (0.75,-3.25) rectangle (1.25,-2.75);
            \draw[thick, black] (1.75,-3.25) rectangle (2.25,-2.75);
            \node at (0,-4.75) {\dots};
            \node [rotate=-45] at (4.75,-4.75) {\vdots};
            \node at (-0.75,0) {};
        \end{tikzpicture}
    \end{center}
    In the modified version, if we let $p(x,y)$ be the entry in row $y$ and column $x$, we have 
    $$
        p(x,y)=p(x-1,y)+p(x-1,y-1).
    $$
    We make two additional modifications:
        \begin{enumerate}
            \item we align Pascal's triangle such that the bottom left entry occurs at $(2,1)$ and
            \item we shade all odd entries.
        \end{enumerate}
    On the one hand, this is well-known to be the discrete Sierpi\'nski triangle $\Sierp$. On the other hand, because only odd entries are shaded, an entry $(x,y)$ is shaded if and only if exactly one of $(x-1,y)$ or $(x-1,y-1)$ is shaded.

    Now, observe that there are only two potential representations for $1^x 0^{x-y}$, those being
        \begin{enumerate}
            \item ${1^y0^{x-y-1}}\concat 0$ and
            \item $1\concat1^{y-1}0^{x-y}$.
        \end{enumerate}
    Thus, $(x,y)$ is in the desired set if and only if exactly one of $(x-1,y)$, $(x-1,y-1)$ is in the set.
\end{proof}

\begin{remark}
This is not the first time that Pascal's triangle and the Sierpi\'nski triangle have shown up in the study of Ulam words. As mentioned previously, Mandelshtam found an analogous pattern for words $0^x10^y1^z$ \cite{Mandelshtam_2022}; even before that, Bade et al. proved that the number of words of length $n$ with one $1$ is the $n$-th term in Gould's sequence\cite{BCLL_2020}---that is, the number of elements in the $(n-1)$-th row of Pascal's triangle.
\end{remark}

There are many consequences of this result. First, it means that determining the set of Ulam words of the form $1^a 0^b$ is quite simple: it is a straightforward iterative procedure. Counting the number of elements of such form is also simple. For example, consider the following simple corollary.

\begin{corollary}\label{corollary: Gould count}
Fix $n \in \Z_{\geq 2}$; let $L_n$ be the set of $a \in \Z_{\geq 1}$ such that $0^a 1^{n-a} \in \Ulam$. Then $\# L_n = G(n)$.
\end{corollary}

\begin{proof}
The $k$-th term of Gould's sequence is the number of odd entries in the $k$-th row of Pascal's triangle. After rotating and shifting, the $k$-th row corresponds to $L_{k + 1}$.
\end{proof}

There are certainly more symmetries lurking within $\mathcal{S}$ that would lead to more theorems and patterns. In particular, we believe that it might be possible to demonstrate the following.
\begin{conjecture}
    Let $u$ be a word of length $n$ of the form $0^a1^{2^k}0^b$ for $a,b,k\in\mathbb{Z}_{\geq1}$. The word $u$ is in $\Ulam$ if and only if $n=1 \mod 2^k$.
\end{conjecture}

We close by giving one more result, which is analogous to Theorem \ref{theorem: bade 3}.

\begin{theorem}
For any $a,b \in \Z$ such that $a + b \geq 1$, $0^a 10101 0^b \in \Ulam$ if and only if $a,b \in 2\Z$ and one is zero.
\end{theorem}

\begin{proof}
We induct on $a + b$. The base cases $a + b = 1,2$ are easily established by pure computation. First, assume that both $a,b \neq 0$. Then there are six possible representations:
    \begin{enumerate}
        \item $0\concat 0^{a-1}101010^{b}$,
        \item $0^a1\concat 01010^b$,
        \item $0^a10\concat 1010^b$,
        \item $0^a101\concat 010^b$,
        \item $0^a1010\concat 10^b$, and
        \item $0^a101010^{b-1}\concat 0$.
    \end{enumerate}
Applying the inductive hypothesis and Theorem \ref{theorem: bade 3}, we find that in all cases either none of these are valid representations, or multiple are simultaneously.

This leaves the case where either $a = 0$ or $b = 0$---without loss of generality, we assume that $a = 0$, since we can always apply the reverse map if needed. We have three possible representations:
    \begin{enumerate}
        \item $1\concat 01010^b$,
        \item $10\concat 1010^b$, and
        \item $101010^{b-1}\concat 0$.
    \end{enumerate}
If $b$ is odd, then (2) and (3) are valid representations due to Theorem \ref{theorem: bade 3}; therefore, this word is not Ulam. If $b$ is even, then (1) is a valid representation, but (2) and (3) are not; therefore, this word is Ulam.
\end{proof}

\section{Density and the Distribution of the Gaps}\label{section: archimedean equidistribution}

We begin by proving that there is a relationship between the size of the average gap between consecutive Ulam words of length $n$, and the density of Ulam words.

\begin{proof}[Proof of Theorem \ref{theorem: tie between conjectures}]
Let $u_1 < u_2 < \ldots < u_{k_n}$ be the (ordered) elements of $\pi(\Ulam_n)$---that is, the integers that are images of the Ulam words of length $n$. The gaps between them are the consecutive differences $g_1 = u_2 - u_1$, $g_2 = u_3 - u_2$, and so on. It is easy to see that $u_1 = \pi(0000\ldots 01) = 1$ and $u_{k_n} = \pi(1111\ldots 10) = 2^n - 2$---this follows from Theorem \ref{theorem: bade 4}, for example. Ergo, the average gap between Ulam words of length $n$ is
    \begin{align*}
        \mu_g(n) := \frac{g_1 + g_2 + \ldots + g_{k_n - 1}}{k_n - 1} &= \frac{u_{k_n} - u_1}{k_n - 1} = \frac{2^n - 3}{k_n - 1} \\
        &= \frac{1}{\rho(n)}\frac{k_n}{k_n - 1} - \frac{3}{k_n + 1},
    \end{align*}
from which we conclude that $\rho(n)^{-1} \asymp \mu_g(n)$. (We already know from Theorem \ref{theorem: linear growth rate} that $k_n \rightarrow\infty$ as $n\rightarrow\infty$.) In particular, if Conjecture \ref{conjecture: gap distribution} is true and $\mu_g(n) \asymp n^{3/10}$, then $\rho(n) \asymp n^{-3/10}$, which is Conjecture \ref{conjecture: growth rate}.
\end{proof}

Notice in particular that for $\rho(n)$ to converge to a non-zero constant, it must be that $\mu_g(n)$ is bounded. But this does not appear to be the case, as evidenced by Figures \ref{fig:average gap} and \ref{fig:relative error}. Instead, as near as we can tell, the order of growth seems to be around $n^{3/10}$. With this in mind, it is perhaps worthwhile to examine the distribution of the gaps more closely.

\begin{figure}[t]
    \centering
    \includegraphics[width=0.5\linewidth]{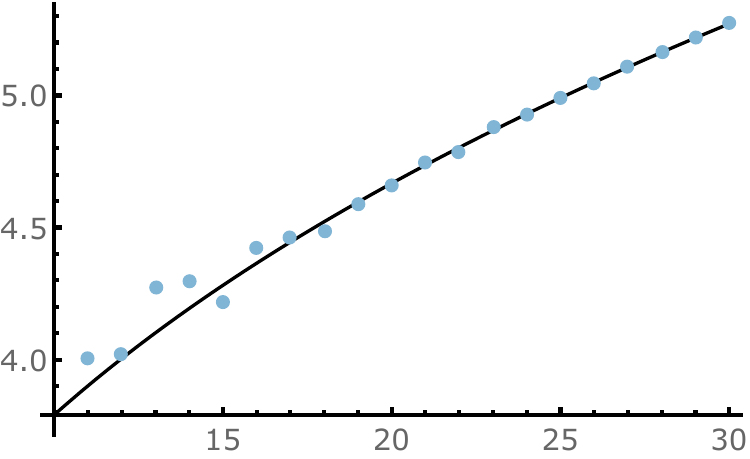}
    \caption{A plot of $\mu_g(n)$ for $11 \leq n \leq 30$ and $f(n) = 1.9 n^{3/10}$.}
    \label{fig:average gap}
\end{figure}

\begin{figure}[t]
    \centering
    \begin{minipage}{0.3\textwidth}
        \begin{tabular}{l|l}
        n & \text{Relative Error} \\ \hline
        13 & 4.08765\% \\
        14 & 2.43579\% \\ 
        15 & 1.54247\% \\ 
        16 & 1.27957\% \\ 
        17 & 0.402629\% \\
        18 & 0.87545\%
    \end{tabular}
    \end{minipage}%
    \begin{minipage}{0.3\textwidth}
        \begin{tabular}{l|l}
        n & \text{Relative Error} \\ \hline
        19 & 0.196991\% \\ 
        20 & 0.222403\% \\ 
        21 & 0.24375\% \\ 
        22 & 0.332677\% \\ 
        23 & 0.33529\% \\
        24 & 0.104678\%
    \end{tabular}
    \end{minipage}%
    \begin{minipage}{0.3\textwidth}
        \begin{tabular}{l|l}
         n & \text{Relative Error} \\ \hline
         25 & 0.0447467\% \\ 
         26 & 0.0981692\% \\ 
         27 & 0.028047\% \\ 
         28 & 0.0353147\% \\ 
         29 & 0.0365254\% \\ 
         30 & 0.00663021\%
    \end{tabular}
    \end{minipage}
    
    \caption{The relative error between the actual $\mu_g(n)$ and the estimate $1.9n^{3/10}$.}
    \label{fig:relative error}
\end{figure}

\begin{figure}[t]
    \centering
    \includegraphics[width=0.5\linewidth]{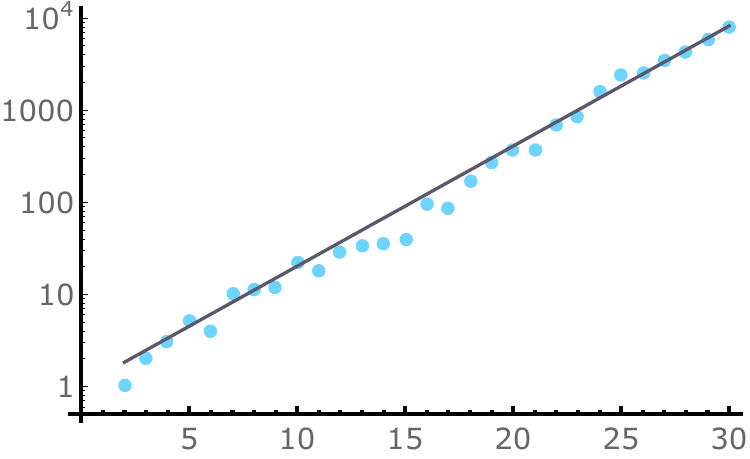}
    \caption{A logarithmic plot of the maximal gap between words of length $n$ for $2 \leq n \leq 30$, and $f(n) = 1.35^n$.}
    \label{fig:maximal gap}
\end{figure}

It is obvious that any gap is at least $1$ (in fact, it is an easy exercise to show that a gap of $1$ is always attained); on the other hand, numerical evidence suggests that the maximal gap grows exponentially---specifically, it is $O(r^n)$ for some constant $r \approx 1.35$ (see Figure \ref{fig:maximal gap}). This is a little strange, in that one would not immediately guess this looking at the probability distributions, as in Figure \ref{fig:gap_distributions}.

\begin{figure}[t]
    \centering
    \begin{tabular}{ccc}
         \includegraphics[width=0.305\linewidth]{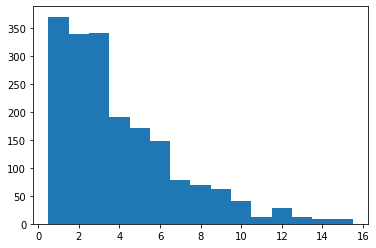} & \includegraphics[width=0.305\linewidth]{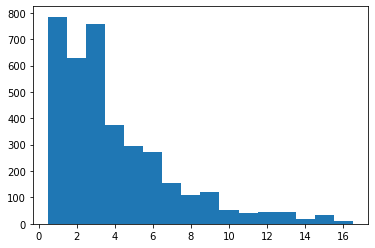} & \includegraphics[width=0.305\linewidth]{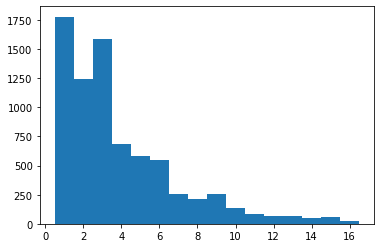} \\
         \includegraphics[width=0.305\linewidth]{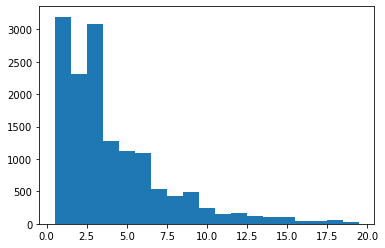} & \includegraphics[width=0.305\linewidth]{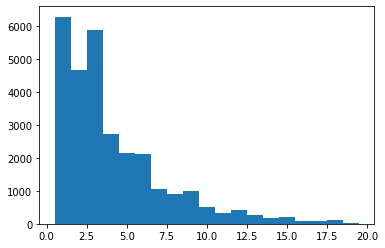} & \includegraphics[width=0.305\linewidth]{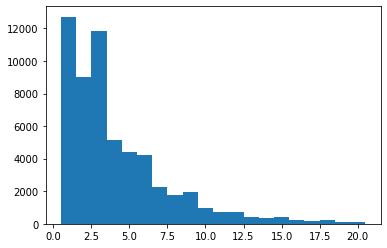} \\
         \includegraphics[width=0.305\linewidth]{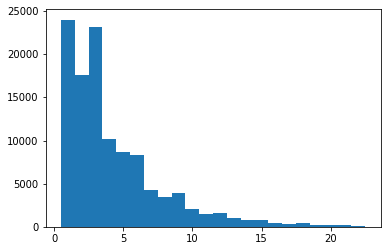} & \includegraphics[width=0.305\linewidth]{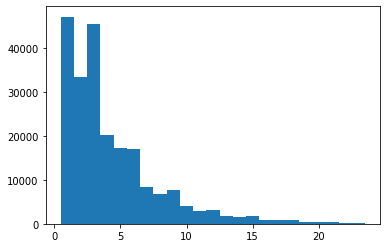} & \includegraphics[width=0.305\linewidth]{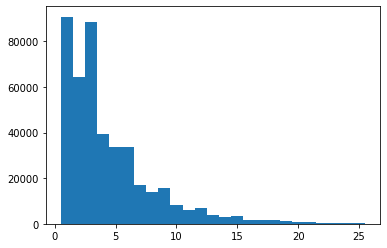} \\
         \includegraphics[width=0.305\linewidth]{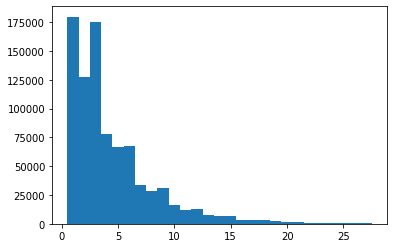} & \includegraphics[width=0.305\linewidth]{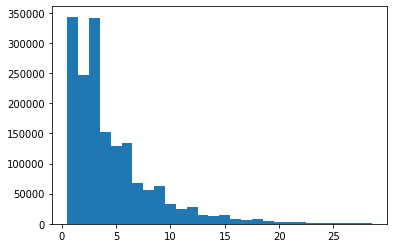} & \includegraphics[width=0.305\linewidth]{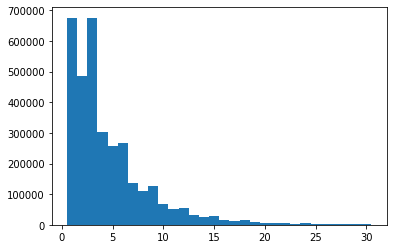} \\
         \includegraphics[width=0.305\linewidth]{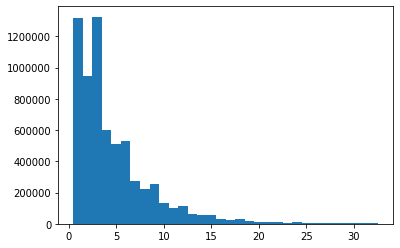} & \includegraphics[width=0.305\linewidth]{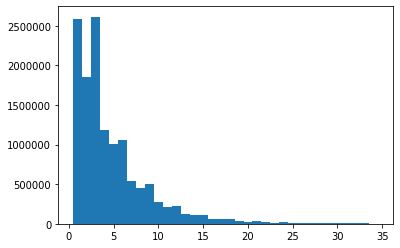} & \includegraphics[width=0.305\linewidth]{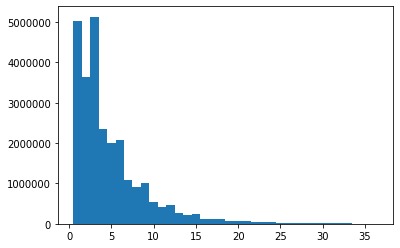} \\
         \includegraphics[width=0.305\linewidth]{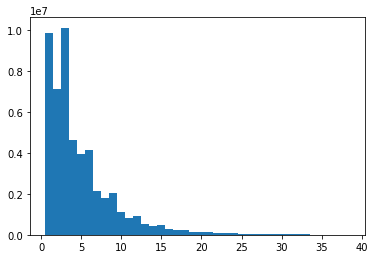} & \includegraphics[width=0.305\linewidth]{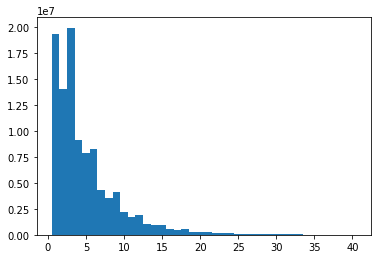} & \includegraphics[width=0.305\linewidth]{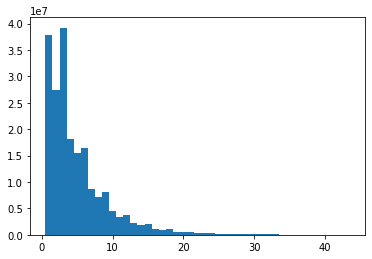}
    \end{tabular}
    
    \caption{From left to right, top to bottom: bar graphs showing the frequency of gaps of various sizes between consecutive words in $\Ulam_n$ for $n = 13,\ldots,30$, shown out to $4$ standard deviations.}
    \label{fig:gap_distributions}
\end{figure}

Indeed, there are some curious details to this apparent distribution. The first is that there is a clear bias against gaps that are either $2$ or $4$ modulo $6$. This is very odd, considering that we conjecture that Ulam words are equidistributed modulo $6$ in the limit. (See Section \ref{section: non-archimedean equidistribution} for more about this.) The second is that just as the average gap appears to grow without bound, so does the standard deviation. However, the standard deviations are quite small, of order around $n/3$---here is a table of the last few we were able to compute:
    $$
    \begin{array}{l|l}
        n & \text{Standard Deviation} \\ \hline
        21 & 6.09043461 \\
        22 & 6.57391412 \\
        23 & 6.95198536 \\
        24 & 7.48652894 \\
        25 & 7.95451379 \\
        26 & 8.41026105 \\
        27 & 8.83842107 \\
        28 & 9.34566047 \\
        29 & 9.94055302 \\
        30 & 10.5007497
    \end{array}
    $$
This means that the distribution is very tightly clustered toward the smaller side. However, it has extreme outliers: the maximal gap between words of length $30$ is $8030$, which is more than $764$ standard deviations away from the mean! Somehow, this should be typical: as we noted already, the size of the maximal gap appears to grow exponentially, but the same is not true of either the average gap or the standard deviation.

\section{Modular Distribution}\label{section: non-archimedean equidistribution}

Let us now consider the relative density of Ulam words. As we mentioned earlier, the set of Ulam words is preserved under the complement map. This forces a symmetry on congruence classes.

\begin{theorem}
If $w \in \Ulam_n$ then $\pi^{-1}\left(2^{n+1} - 1 - \pi(w)\right) \in \Ulam_n$. Consequently, for any positive integer $N$ and $a \in \Z/N\Z$,
    $$
    \rho_{a,N}(n) = \rho_{2^{n+1}-1-a,N}(a).
    $$
\end{theorem}

\begin{proof}
Given $w \in \Ulam_n$, write $x = \pi(w) = a_n 2^n + \ldots + a_0$ in binary. Then
    $$
    \pi\left(\hat{w}\right) = (1-a_n)2^n + \ldots + (1 - a_0) = 2^{n + 1} - 1 - x.
    $$
But $\hat{w} \in \Ulam_n$. Now, observe that if $\pi(w) = a \mod N$, then $2^{n+1} - 1 - \pi(w) = 2^{n + 1} - 1 - a \mod N$, which forces the equality of the relative densities.
\end{proof}

For $N = 2,3$, this is particularly simple.

\begin{corollary}
For any positive integer $n$, $\rho_{0,2}(n) = \rho_{1,2}(n)$. Furthermore, for any $a \in \Z/3\Z$,
    $$
    \rho_{a,3}(n) = \begin{cases}
        \rho_{1-a,3}(n) & \text{if } n = 0 \mod 2 \\
        \rho_{-a,3}(n) & \text{if } n = 1 \mod 2.
    \end{cases}
    $$
\end{corollary}

\begin{proof}
Observe that $2^{n+1}-1-x = x + 1 \mod 2$, from which $\rho_{0,2}(n) = \rho_{1,2}(n)$ follows immediately. For the second part, observe that
    $$
    2^{n+1} - 1 - x \mod 3 = \begin{cases}
        1 - x & \text{if } n = 0 \mod 2\\
        -x & \text{if } n = 1 \mod 2.
    \end{cases}
    $$
\end{proof}

While there must always exist for any $n$ two congruence classes $a,b \in \Z/3\Z$ such that $\rho_{a,3}(n) = \rho_{b,3}(n)$, there is no reason why the last congruence class $c$ should be roughly equal. Indeed, for $n \leq 5$, $\rho_{c,3}(n) = 0$. However, for larger $n$, it does appear to be the case that $\rho_{c,3}(n) \rightarrow \rho_{a,3}(n) = \rho_{b,3}(n)$; as we will illustrate presently. To help measure the extent to which words are equidistributing modulo $N$, we define the modular discrepancy.


\begin{definition}
For any positive integers $n,N$, the \emph{modular discrepancy} is
    $$
    d_N(n) := \max_{a,b \in \Z/N\Z} \left|\rho_{a,N}(n) - \rho_{b,N}(n)\right|.
    $$
\end{definition}

Trivially, saying that Ulam words equidistribute modulo $N$ is equivalent to saying that $d_N(n) \rightarrow 0$ as $n \rightarrow \infty$. Moreover, by appealing to the Chinese remainder theorem, proving that $d_N(n) \rightarrow 0$ for all $N$ is reducible to proving that $d_{p^k}(n) \rightarrow 0$ for all prime powers $p^k$. To investigate Conjecture \ref{conjecture: nonarchimedean equidistribution}, we computed $d_{p^k}(n)$ for all prime powers $p^k < 30$---as near as we can tell, $d_{p^k}(n)$ decays exponentially as a function of $n$ (see Figure \ref{fig:modular_discrepancy_plot}.)

\begin{figure}[t]
    \centering
    \includegraphics[width=0.6\linewidth]{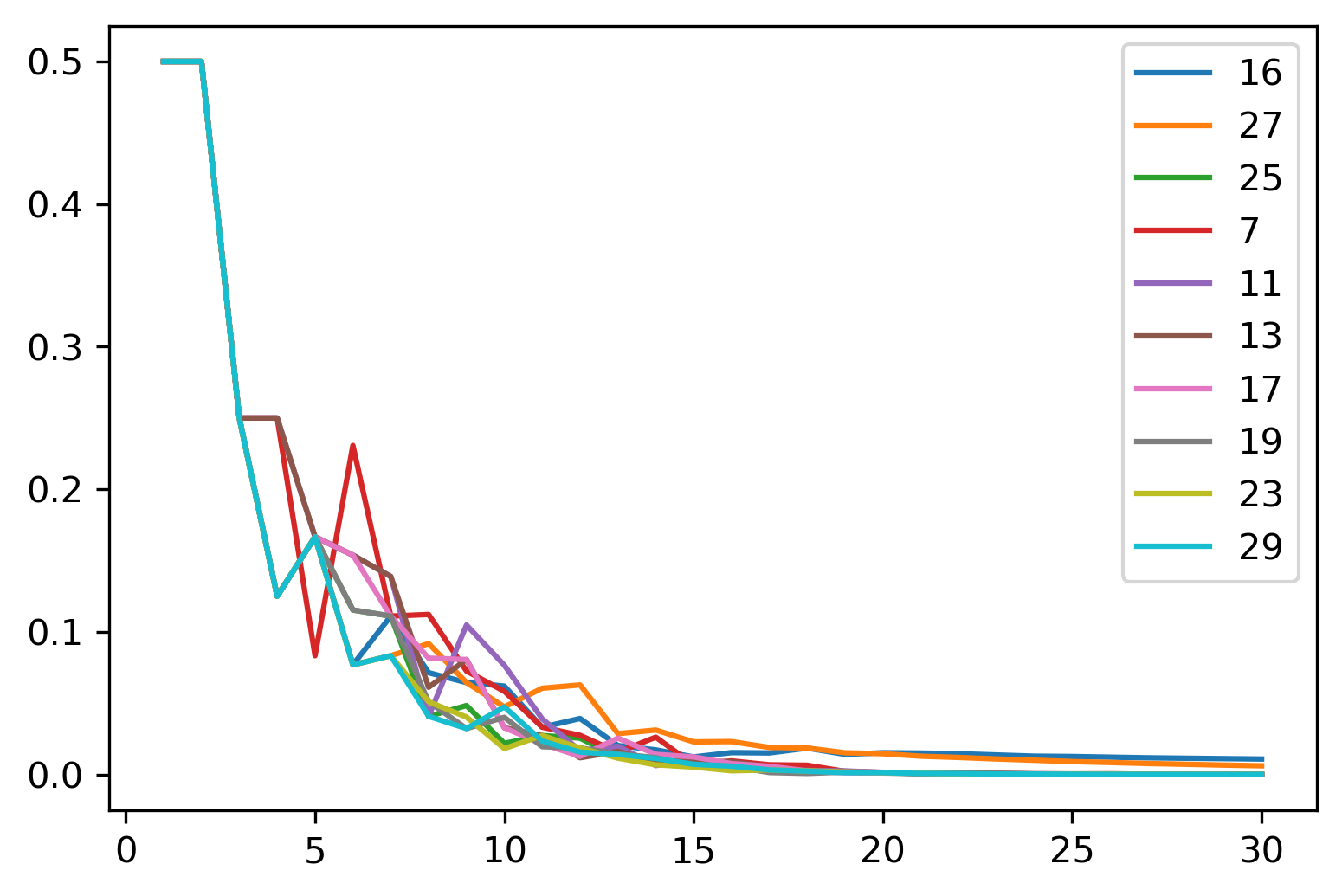}
    \caption{A plot of the modular discrepancies for prime power moduli $p^k < 30$.}
    \label{fig:modular_discrepancy_plot}
\end{figure}

\subsection*{Acknowledgments.}
Our collaboration was funded by four Bates College grants, all awarded by the Dean of Faculty's office: a STEM Faculty-Student Summer Research Grant and three Summer Research Fellowships. We would also like to thank Tom\'as Oliveira e Silva for confirming our computations of $\#\Ulam_n$ and giving many helpful suggestions for improving the exposition.

\bibliography{References}
\bibliographystyle{plain}

\end{document}